\documentclass[11pt]{amsart}
\usepackage{fullpage}
\usepackage{graphicx}
\usepackage{amsmath}
\usepackage{amsthm,amsfonts,amssymb,mathrsfs,amscd,amstext,amsbsy}
\usepackage{epic,eepic}
\usepackage{enumerate}
\usepackage[all]{xy}
\usepackage{tikz}
\usepackage{color}

\usetikzlibrary{patterns}

\newtheorem{theorem}{Theorem}[section]

\newtheorem{claim}[theorem]{Claim}
\newtheorem{thm}[theorem]{Theorem}

\def\qed{\hfill \ifhmode\unskip\nobreak\fi\quad\ifmmode\Box\else$\Box$\fi\\ }

\newcommand\abs[1]{\lvert #1\rvert}

\begin{document}
\title{On the Erd\H os-Ko-Rado  Theorem and the Bollob\'{a}s  Theorem \\ for $t$-intersecting families}

\author{Dong Yeap Kang, Jaehoon Kim, and Younjin Kim
 }
\address[Dong Yeap Kang]{Department of Mathematical Sciences, KAIST, 291 Daehak-ro Yuseong-gu Daejeon, 305-701 South Korea}
\address[Jaehoon Kim]{Department of Mathematical Sciences, KAIST, 291 Daehak-ro Yuseong-gu Daejeon, 305-701 South Korea / School of Mathematics, University of Birmingham, Edgbaston, Birmingham B15 2TT, United Kingdom}
\address[Younjin Kim]{Department of Mathematical Sciences, KAIST, 291 Daehak-ro Yuseong-gu Daejeon, 305-701 South Korea}
\email{dynamical@kaist.ac.kr}
\email{kimjs@bham.ac.uk, mutualteon@gmail.com}
\email{younjin@kaist.ac.kr}

\thanks{ The first author and second author were partially supported by Basic Science Research Program through the National Research Foundation of Korea(NRF) funded by the Ministry of
Science, ICT \& Future Planning (2011-0011653). \\
The second author was also partially supported by the European Research Council under the European Union's Seventh Framework Programme (FP/2007--2013) / ERC Grant Agreements no. 306349 (J.~Kim). \\
The third author was supported by Basic Science Research Program through the National Research Foundation of Korea(NRF) funded by the Ministry of
Science, ICT \& Future Planning (2013R1A1A3010982) (Y.~Kim).\\
Corresponding author: Younjin Kim}

\date{\today}
\begin{abstract}
A family $\mathcal{F}$ is $t$-$\it{intersecting}$ if any two members have at least $t$ common elements. Erd\H os, Ko, and Rado~\cite{EKR} proved that the maximum size of a $t$-intersecting family of subsets of size $k$ is equal to $ {{n-t} \choose {k-t}}$ if $n\geq n_0(k,t)$.
Alon, Aydinian, and Huang~\cite{ALON} considered families generalizing intersecting families, and proved the same bound. In this paper, we give a strengthening of their result by considering families generalizing $t$-intersecting families for all $t \geq 1$.
 In 2004, Talbot~\cite{TAL} generalized Bollob\'{a}s's Two Families Theorem~\cite{BOL} to $t$-intersecting families. In this paper, we proved a slight generalization of Talbot's result by using the probabilistic method.
 \end{abstract}
 
\maketitle

\section {Introduction}
Let $[n]$ be the set $\{ 1, \ldots, n \} $  and  $t$ be a positive integer.
A family $\mathcal{F}$ of subsets of $[n] $ is $t$-intersecting if $|F_i \cap F_j | \geq t $ for every pair of two subsets $F_i,F_j \in \mathcal{F}$. A family $\mathcal{F}$ of subsets of $[n] $ is $k$-uniform if it is a collection of $k$-subsets of $[n]$, and we also say that it has rank $k$ if the largest set in it has size $k$. Erd\H os, Ko, and Rado~\cite{EKR} proved that there exists $n_0(k,t)$ such that if $n \geq n_0(k,t)$, then the maximum size of a $k$-uniform $t$-intersecting family of subsets of $[n]$ is  ${{n-t} \choose {k-t}}$. The following generalization of Erd\H os, Ko, and Rado (EKR) Theorem was proved by Frankl~\cite{FR78} for $t \geq 15$, and  was completed by Wilson~\cite{WIL} for all $t$ by obtaining the smallest $n$ for the theorem to be true. \\

\begin{thm} [Frankl~\cite{FR78}, Wilson~\cite{WIL}]
If $\mathcal{F}$ is a $k$-uniform $t$-intersecting family of subsets of $[n]$, then we have
$$ |\mathcal{F}| \leq {{n-t} \choose {k-t}}$$ whenever $n \geq (k-t+1)(t+1)$.\\

\end{thm} 

There is also an EKR-type theorem for $t$-intersecting families with bounded rank: if $n \geq (k-t+1)(t+1)$, then the maximum size of a $t$-intersecting family of subsets of sizes at most $k$ is equal to $ {{n-t} \choose {k-t}} + {{n-t} \choose {k-1-t}} + \cdots + {{n-t} \choose 0}$. 
 In this paper we give the following strengthening of this theorem.

\begin{thm}\label{thm:t_intersecting}
Let $n, k, t$ be integers such that $n \geq (k-t+1)(t+1)$.
Let a collection of sets $\mathcal{F} \subseteq {{[n]}\choose {\leq k}}$ satisfies the following.
Suppose that for any $A,B \in \mathcal{F}$ with $|A\cap B| <t$, $|A\triangle B| \leq k-t$ holds.
Then  we have $$|\mathcal{F}|\leq {{n-t} \choose {k-t}} + {{n-t} \choose {k-1-t}} + \cdots + {{n-t} \choose 0} .$$ 
\end{thm}

\noindent Alon, Aydinian, and Huang~\cite{ALON} gave the following strengthening of the bounded rank EKR theorem when $t=1$:  $|\mathcal{F}|\leq {{n-t} \choose {k-t}} + {{n-t} \choose {k-1-t}} + \cdots + {{n-t} \choose 0}$ holds for all pairs $A,B$ with $|A\cap B|<1$ when $|A\triangle B| \leq k$, instead of $|A\triangle B|\leq k-1$. However, the condition $|A\triangle B| \leq k-t$ in Theorem \ref{thm:t_intersecting} cannot be replaced with $|A\triangle B| \leq k-t+1$ for $t\geq 2$ because of the following example.

\noindent For $t<k<n$, consider $\mathcal{F}=\{A\subseteq {{[n]}\choose {\leq k}} : [t]\subseteq A\} \cup \{A\subseteq [t]:|A|=t-1\}$. Then $\mathcal{F}$ contains ${{n-t} \choose {k-t}} + {{n-t} \choose {k-1-t}} + \cdots + {{n-t} \choose 0} + t$ sets while every two sets with $|A\cap B|\leq t-1$ satisfies $|A\triangle B| \leq k-t+1$. Thus our condition $|A\triangle B| \leq k-t$ is best possilbe when $t\geq 2$. \\

We also extend results regarding cross $t$-intersecting families in the same manner. We say that families $\mathcal{F}_1, \mathcal{F}_2, \cdots, \mathcal{F}_r$ of subsets of $[n]$ are cross $t$-intersecting if $|A\cap B| \geq t $ for every $A \in \mathcal{F}_i$ and $B \in \mathcal{F}_j$, where $i \neq j$. In 2013, Borg~\cite{BORG13} obtained the maximum product of sizes of cross $t$-intersecting families with bounded rank as follows. In this paper, we also show Theorem \ref{thm:cross}, a strengthening of Theorem \ref{thm:borg}.

\begin{thm}[Borg~\cite{BORG13}]\label{thm:borg}
For $ 1 \leq t \leq k \leq n $, there exists $n_0(k,t)$ such that for all $n \geq n_0(k,t)$ the maximum size of product of families $\mathcal{F}_i \subseteq {{[n]} \choose {\leq k_i}}$ for $1\leq i \leq r$ is equal to 
$$ \prod_{i=1}^{r}\left(\sum_{j=0}^{k_i-t} { {n-t} \choose {k_i-t-j}} \right) $$ when $\mathcal{F}_1, \mathcal{F}_2, \cdots, \mathcal{F}_r$ are cross $t$-intersecting families.

\end{thm}

\begin{thm}\label{thm:cross}
For $ 1 \leq t \leq k_1\leq k_2\leq \cdots \leq k_r \leq n$, there exists $n_0(k_{r-1},k_{r},t)$ such that the following holds for all $n\geq n_0(k_{r-1},k_{r},t)$. Let  $\mathcal{F}_1, \mathcal{F}_2, \cdots, \mathcal{F}_r$ be families of subsets of $[n]$ of size at most $k_1,k_2,\ldots,k_r$, respectively. 
 If every $A \in \mathcal{F}_i $ and  $B \in \mathcal{F}_j$ with $|A \cap B| <t $ satisfies $|A\triangle B| \leq \min\{k_i,k_j\} -t$, then we have $$ \prod_{i=1}^r |\mathcal{F}_i|  \leq \prod_{i=1}^{r}\left(\sum_{j=0}^{k_i-t} { {n-t} \choose {k_i-t-j}} \right).$$
\end{thm}

Next, we consider $r$-wise intersecting families. We say that a family $\mathcal{F} \subseteq {{[n]} \choose k}$ is $r$-wise intersecting if $F_1 \cap F_2 \cap \cdots \cap F_r \neq \emptyset $ holds for all $F_i \in \mathcal{F}, 1 \leq i \leq r$. 
In the following Theorem~\ref{cor1}, we give a strengthening of the bounded-rank version of  Frankl's result~\cite{FR76} when $\mathcal{F}$ is an $r$-wise intersecting family.

\begin{thm}[Frankl~\cite{FR76}]\label{f}
Let $(r-1)n \geq rk$ and $\mathcal{F} \subseteq {{[n]}\choose { k}}$ be  an $r$-wise intersecting family. 
Then  we have $$|\mathcal{F}|\leq {{n-1}\choose {k-1}} .$$
\end{thm}

\begin{thm}\label{cor1}
Let $(r-1)n \geq rk$ and $\mathcal{F} \subseteq {{[n]}\choose {\leq k}}$ satisfies the following.
Suppose that for any $F_1, F_2, \cdots, F_r \in \mathcal{F}$ if $ F_1\cap F_2\cap \cdots \cap F_r = \emptyset$, $|\cup_{i=1}^r F_i - \cap_{i=1}^rF_i| \leq k$ holds.
Then  we have $$|\mathcal{F}|\leq {{n-1}\choose {k-1}} + {{n-1}\choose {k-2}}+ \cdots {{n-1}\choose 0}.$$
\end{thm}

The following Bollob\'{a}s's Two Families Theorem~\cite{BOL} is an important and well-known result in Extremal Set theory.  This theorem has been generalized in many directions. In 1982, Frankl~\cite{FR82} proved a skew version of Theorem~\ref{thm:bollobas}. The further generalizations of  Theorem~\ref{thm:bollobas}  were given by F\"{u}redi~\cite{FUREDI},  Lov\'{a}sz~\cite{LOVA}, and Talbot~\cite{TAL}.
In 2004, Talbot~\cite{TAL}  generalized Bollob\'{a}s's two families theorem~\cite{BOL} to $t$-intersecting families. We prove Theorem~\ref{thm:improve}, an improvement of Talbot's result, by using the probabilistic method. It also becomes to reprove Talbot's Theorem~\ref{thm:talbot}.  Note that the conditions $(c')$ in Theorem~\ref{thm:improve} is weaker than the condition $(c)$ of Theorem \ref{thm:talbot}.\\

\begin{thm}[Bollob\'{a}s~\cite{BOL}]\label{thm:bollobas}
Let $\mathcal{F}= \{ (A_i, B_i ) : i \in I\}$ be a finite collection of pairs of finite sets such that $A_i \cap B_j = \emptyset$ $\Longleftrightarrow$ $i=j$. Then we have
$$ \sum_{i \in I} {{|A_i|+|B_i|} \choose {|A_i|}}^{-1} \leq 1.$$
\end{thm}

\begin{thm}[Talbot~\cite{TAL}] \label{thm:talbot}
Let $\mathcal{F}= \{ (A_i, B_i ) : i \in I\}$ be a finite collection of pairs of finite sets and $t$ be a nonnegative integer such that
 \begin{align*} & (a) \  \abs{A_i \cap B_i} \leq t \ \ {\text \ for \ each
\ }  i \in I
\\
 & (b) \   \abs{A_i \cap B_j} \geq t  \ \ {\text \ for
\ } i, j \in I  {\text \ and \ } i \neq j
 \\
 & (c) \  {\text If \ } A_i \cap B_i = A_j \cap B_j { \text  \ for \  } i \neq j \ \
{\text \ then \ } A_i \cap B_j \neq A_i \cap B_i \neq A_j \cap B_i.
 \end{align*}
Then we have
\begin{eqnarray*}
\sum_{i \in I}{\binom{|A_i \cup B_i |}{|A_i - B_i |}^{-1} \binom{|B_i |}{|A_i \cap B_i |}^{-1}} \leq 1.
\end{eqnarray*}
\end{thm}

\begin{thm} \label{thm:improve} \label{thm:improve} When $t\geq 1$, even if we replace the condition $(c)$  in Theorem~\ref{thm:talbot} with the following weaker condition
\begin{align*}
&(c') \  {\text If \ } A_i \cap B_i = A_j \cap B_j { \text  \ for \  } i \neq j \ \
{\text \ then \ } A_i \cap B_j = A_i \cap B_i = A_j \cap B_i \text{ does not hold} \end{align*}
we still have
\begin{eqnarray*}
\sum_{i \in I}{\binom{|A_i \cup B_i |}{|A_i - B_i |}^{-1} \binom{|B_i |}{|A_i \cap B_i |}^{-1}} \leq 1.
\end{eqnarray*}

\end{thm}

\noindent Note that the above conditions in Theorem \ref{thm:improve} are sharp in the sense that we cannot replace $(c)$ with $(c')$ when $t=0$ without changing the conclusion because of the following theorem.\\

\begin{thm}(Kir\'{a}ly, Nagy, P\'{a}lv\H{o}lgyi and Visontai \cite{KNPV}) \label{examp}
Let $a,b$ be two relatively prime integers. There exists a finite collection of pairs of finite sets $\mathcal{F}= \{ (A_i, B_i ) : |A_i|=a, |B_i|=b,i \in I\}$ with $A_i\cap B_i = \emptyset$ for all $i\in I$. Moreover, $ A_i \cap B_j = A_i \cap B_i = A_j \cap B_i =\emptyset\text{ does not hold }$ for any $i\neq j$ and
 $$\sum_{i \in I}{\binom{|A_i \cup B_i |}{|A_i - B_i |}^{-1} \binom{|B_i |}{|A_i \cap B_i |}^{-1}} = |I|\binom{a+b}{a}^{-1} \geq 2-\frac{1}{a+b}.$$

\end{thm}

\section{Proof of Theorems}
\noindent For a family $\mathcal{F}$ of subsets of $[n]$ of size at most $k$, we define  the transformation $S_{i,k}$ as follows.

  \begin{displaymath}
S_{i,k,\mathcal{F}}(A):=\left\{
\begin{array}{ll}
A\cup \{i\} & \mbox{ if $A\cup \{i\}\notin \mathcal{F}$ and $|A|<k$} \\
A & \mbox{otherwise}\\
\end{array}
 \right.
 \end{displaymath}

\noindent Also, let $S_{i,k}(\mathcal{F}) = \{ S_{i,k,\mathcal{F}}(A) : A\in \mathcal{F}\}$. \\

\begin{claim} \label{c1}
A tuple of families $(\mathcal{F}_1,\mathcal{F}_2,\cdots,\mathcal{F}_r)$ is given with $\mathcal{F}_j \subseteq {{[n]}\choose {\leq k_j}}$ for $j=1,\ldots,r$ and $k = \min_j\{ k_j\}$ satisfying the following 
\begin{equation} \label{1}
 \text{If } |\bigcap_{j=1}^{r} F_j| < t \text{ for } F_j \in \mathcal{F}_j, \text{ then } |\bigcup_{j=1}^{r} F_j - \bigcap_{j=1}^{r} F_j| \leq k-t \text{ holds. }
\end{equation}
Then, for $i\in [n]$, $(S_{i,k_1}(\mathcal{F}_1), S_{i,k_2}(\mathcal{F}_2), \cdots, S_{i,k_r}(\mathcal{F}_r))$ also satisfies (\ref{1}).
\end{claim}
\begin{proof}{
Suppose that we have $F_j\in \mathcal{F}_j$ such that  $F'_j=S_{i,k,\mathcal{F}_j}(F_j)$ with $|\bigcap_{j=1}^{r} F'_j|<t$ and $| \bigcup_{j=1}^{r} F'_j -\bigcap_{j=1}^{r} F'_j| > k-t$.
Since $\bigcup_{j=1}^{r} F_j -\bigcap_{j=1}^{r} F_j$ is a subset of $\bigcup_{j=1}^{r} F'_j -\bigcap_{j=1}^{r} F'_j$ with size at most $k-t$, we have $$\bigcup_{j=1}^{r} F'_j -\bigcap_{j=1}^{r} F'_j= \{i\}\cup \bigcup_{j=1}^{r} F_j -\bigcap_{j=1}^{r} F_j .$$ Hence $i$ does not belong to any of $F_j$ and belongs to some, not all of $F'_j$.
Assume that $i$ does not belong to $F'_1,\cdots,F'_l$ and belongs to $F'_{l+1},\cdots, F'_{r}$. 
\vspace{0.2cm}

\noindent {\bf {Case 1.} There exists $\mathbf{j}$ with $\mathbf{|F_j|\geq k}$.}

\vspace{0.2cm}

 \noindent Since $|\bigcap_{j=1}^{r} F_j|<t$, we have $$|\bigcup_{j=1}^{r} F_j - \bigcap_{j=1}^{r} F_j| \geq |F_j - \bigcap_{j=1}^{r} F_j |> k-t.$$ It is a contradiction. 
\vspace{0.2cm}

\noindent {\bf{Case 2.} $\mathbf{|F_j|<k}$ \bf{for all} $\mathbf{j}$.} 

\vspace{0.2cm}

\noindent For $j\geq l+1$,  we have $S_{i,k,\mathcal{F}}(F_j)=F'_j=F_j\cup \{i\}$ because of  $i\in F'_j$. For $j\leq l$, we also have $S_{i,k,\mathcal{F}}(F_j)=F'_j=F_j$ because of  $i\notin F'_j$. The condition $k\leq k_j$ implies that $F_j\cup \{i\}\in \mathcal{F}_j$ for $j\leq l$.
Then we have  

$\displaystyle| \bigcap_{j=1}^{r-1} F_j\cap (F_r \cup \{i\})|<t$ and $\displaystyle |(F_r\cup \{i\})\cup \bigcup_{j=1}^{r-1} F_j -  ( (F_r\cup\{i\})\cap \bigcap_{j=1}^{r-1}F_j)| =k-t+1.$  

\noindent It is a contradiction.
Hence such $F_j$s do not exist, and  $(S_{i,k_1}(\mathcal{F}_1), S_{i,k_2}(\mathcal{F}_2), \cdots, S_{i,k_r}(\mathcal{F}_r))$ also satisfies (\ref{1}).
}\end{proof}

\noindent We say  that a family $\mathcal{F} \subseteq {{[n]}\choose {\leq k}}$ is an up-set if $A\cup \{i\} \in\mathcal{F}$ for all $A$ and $i$ such that $A\in \mathcal{F} \cap {{[n]}\choose {\leq k-1}}$ and $i \notin A$.

\begin{claim} \label{c2}
A tuple of families $(\mathcal{F}_1,\mathcal{F}_2,\cdots,\mathcal{F}_r)$ is given with $\mathcal{F}_j \subseteq {{[n]}\choose {\leq k_j}}$ for $j=1,\ldots,r$ and $k = \min_j\{ k_j\}$. If it satisfies (\ref{1}) and $\mathcal{F}_j$s are up-sets for all $j$, then $|\bigcap_{j=1}^{r} F_j| \geq t$ holds for any $F_j\in \mathcal{F}_j$.
\end{claim}\begin{proof}{
Suppose that we have $F_j \in \mathcal{F}_j$ for $j=1,2,\cdots,r$ with $|\bigcap_{j=1}^{r}F_j|<t$. By the condition, $| \bigcup_{j=1}^{r} F_j - \bigcap_{j=1}^{r} F_j|\leq k-t$ holds.
It means $|\bigcup_{j=1}^{r}F_j|<k-1$, so $[n]-\bigcup_{j=1}^{r} F_j$ is not empty because of $n\geq k$. Let $[n]-\bigcup_{j=1}^{r} F_j = \{i_1,i_2,\cdots,i_l\}$.
Since $\mathcal{F}_1$ is an up-set, $F_1\cup \{i_1\}$ is also in $\mathcal{F}_1$ unless $|F_1|=k_1$. By repeatedly adding this element $i_1, i_2,\ldots$ to $F_1$ until we cannot, we find a new set $F''_1 \in \mathcal{F}_1$ satisfying $|F''_1|=k_1$ or $F''_1 = F_1 \cup ([n]-\bigcup_{j=1}^{r} F_j)$.
In any case,  we have $$ \left| \left(F''_1\cup \bigcup_{j=2}^{r} F_j \right) - \left(F''_1\cap \bigcap_{j=2}^{r} F_j \right) \right|\geq \min\{k_1-(t-1), n-(t-1)\}\geq k-t+1.$$ It is a contradiction.
}\end{proof}

 \noindent For a family $\mathcal{F}\subseteq {{[n]}\choose{\leq k}}$, we let $\mathcal{F}^0=\mathcal{F}$ and take $\mathcal{F}^{i+1}=S_{n,k}(S_{n-1,k}(\cdots(S_{1,k}(\mathcal{F}^{i})\cdots)$ for  all $i \geq 0$. If $\mathcal{F}^{j+1} = \mathcal{F}^j$, then we get an up-set $\mathcal{F}^j$ and we let $\mathcal{F}'=\mathcal{F}^j$. We also have  $|\mathcal{F}|=|\mathcal{F}^i|$ for all $i \geq 0$
 since $S_{i,k}$ does not change the size of a family in ${{[n]}\choose {\leq k} }$.

\begin{proof}[Proof of Theorem \ref{thm:t_intersecting}]
{Let  $n \geq (k-t+1)(t+1)$.
We take $\mathcal{F}'$ as above which has the same size with $\mathcal{F}$ and which is an up-set. By  Claim \ref{c1} with $\mathcal{F}_1=\mathcal{F}_2=\mathcal{F}$ and $r=2$, we say that the up-set $\mathcal{F}'$ has the property (\ref{1}) as $\mathcal{F}$ does. Then, Claim \ref{c2} with $\mathcal{F}_1=\mathcal{F}_2=\mathcal{F}$ implies that the up-set $\mathcal{F}'$ is $t$-intersecting. Hence, we have that $$|\mathcal{F}|=|\mathcal{F}'| \leq \sum_{i=0}^{k-t} {{n-t}\choose {k-i-t}}.$$}\end{proof}

\begin{proof}[Proof of Theorem \ref{thm:cross}]
  Let  $ 1 \leq t \leq k_1\leq k_2\leq \cdots \leq k_r \leq n$  and $n \geq n_0(k_{r-1},k_r,t)$ with $n_0(k_{r-1},k_r,t)$ as in Theorem 1.2 in \cite{BORG13}. 
 For each  $\mathcal{F}_i \subseteq {{[n]}\choose {\leq k_i}}$, we define $\mathcal{F}'_i$ by repeatedly applying $S_{j,k_i}$ for all $j$ as above.
Then $\mathcal{F}'_1, \mathcal{F}'_2, \cdots, \mathcal{F}'_r$ are up-sets satisfying (\ref{1}) having the same size with $\mathcal{F}_1, \mathcal{F}_2, \cdots, \mathcal{F}_r$, respectively. By Claim \ref{c1},
we have $|A \triangle B| \leq \min\{k_i,k_j\} -t$ for any $A \in \mathcal{F}'_i , B \in \mathcal{F}'_j$ with $|A \cap B|  < t$. 
For all possible pair of two distinct numbers $j_1,j_2 \in [n]$, we apply Claim \ref{c2} with $\mathcal{F}'_{j_1},\mathcal{F}'_{j_2}$. Then we can conclude that $\mathcal{F}'_1, \mathcal{F}'_2,\cdots, \mathcal{F}'_r$ are cross $t$-intersecting. By Theorem \ref{thm:borg}, we have that $$\prod_{i=1}^{r}|\mathcal{F}_i|=\prod_{i=1}^{r}|\mathcal{F}'_i| \leq \prod_{i=1}^{r}\left(\sum_{j=0}^{k_i-t} { {n-t} \choose {k_i-t-j}} \right).$$
\end{proof}

\begin{proof}[Proof of Theorem \ref{cor1}]
Let  $n \geq (k-t+1)(t+1)$.
We take $\mathcal{F}'$ as above which has the same size with $\mathcal{F}$ and which is an up-set. By  Claim \ref{c1} with $\mathcal{F}_1=\mathcal{F}_2=\cdots =\mathcal{F}_r=\mathcal{F}$, we say that an up-set $\mathcal{F}'$ has the property (\ref{1}) with $t=1$ as $\mathcal{F}$ does. Then, Claim \ref{c2} with $\mathcal{F}_1=\mathcal{F}_2=\cdots=\mathcal{F}_r=\mathcal{F}$ implies that the up-set $\mathcal{F}'$ is intersecting. We let $\mathcal{F}'[i]$ denote the collection of all sets of size $i$ in $\mathcal{F}'$, then $\mathcal{F}'[i]$ is also intersecting. Hence, Theorem \ref{f} implies $|\mathcal{F}'[i]| \leq {{n-1}\choose{i-1}}$.

\noindent Hence, we conclude that $$|\mathcal{F}|=|\mathcal{F}'|=  \sum_{i=1}^{k} |\mathcal{F}'[i]| \leq \sum_{i=1}^{k} {{n-1}\choose {i-1}}.$$
\end{proof}

\section{A probablistic proof of Theorem \ref{thm:talbot} and Theorem \ref{thm:improve}}

In this section, we prove both Theorem \ref{thm:talbot} and Theorem \ref{thm:improve} in the probabilistic argument. Let $\mathcal{F}= \{(A_i,B_i): i\in I\}$ be a finite collection of pairs of finite sets satisfying the conditions on Theorems.  Let $[n] = \bigcup_{i \in I}{A_i \cup B_i}$ and $|I|=k$. 
A permutation $\sigma = (x_1 , x_2 , ... , x_n )$ on $[n]$ is \emph{properly separating} $(A_i , B_i)$, where $i \in I$, if $x_s \in A_i-B_i$ and $x_l \in B_i$ implies $s < l$ and $x_s \in A_i$ and $x_l \in B_i-A_i$ also implies $s< l$. Then all elements of $A_i - B_i$ should be on the left side of $B_i$, and all elements of $B_i - A_i$ should be on the right side of $A_i$ in the permutation $\sigma$ that properly separates $(A_i , B_i)$.

\begin{claim}
Any permutation $\sigma \in S_n$ does not properly separate both $(A_i , B_i)$ and $(A_j , B_j)$, $i \ne j$.
\end{claim}

\begin{proof}
Suppose $\sigma = (x_1 , x_2 , ... , x_n)$ properly separates both $(A_i , B_i)$ and $(A_j , B_j)$, where $i \neq j$. Note that $\max\{i | x_i\in A_i\} \leq \max\{i |x_i\in B_i\}$ by the definition of proper separation.

\noindent We define the following indices $$i_{b} = \min \left \{ s \: | \: x_s \in B_i \right \}, i_{a} = \max \left \{ s \: | \: x_s \in A_i \right \}, j_{b} = \min \left \{ l \: | \: x_l \in  B_j \right \}, j_{a} = \max \left \{ l \: | \: x_l \in A_j \right \}.$$ 
Without loss of generality, we may assume $j_{b} \leq i_{b}$. \\




\noindent {\bf {Case 1.} $t=0$.} 
\vspace{0.2cm}

\noindent Since $A_j\cap B_j =\emptyset$ and $\sigma$ properly separates $(A_j,B_j)$, we have $j_{a}<j_{b}$. Thus $j_{a} < j_{b} \leq i_{b}$, so we conclude that $A_j\cap B_i =\emptyset$. It is a contradiction to the condition $(c)$. \\
\vspace{0.2cm}

\noindent {\bf {Case 2.} $t\geq 1$.} 
\vspace{0.2cm}
\noindent Since $j_b \leq i_b$, we have
\begin{eqnarray*}
A_j \cap B_i & = \left \{ x_l \: | \: x_l \in A_j \:, \:    l\geq i_{b} \right \} \subseteq \left \{ x_l \: | \: x_l \in A_j \:, \:   l\geq j_{b} \right \} = A_j\cap B_j.
\end{eqnarray*}
This with the conditions $(a)$ and $(b)$ imply that 
$$ |A_j\cap B_i| =|A_j\cap B_j| =t \geq 1.$$
Since $A_j\cap B_j$ is nonempty, $x_{j_{b}} \in A_j\cap B_j = A_j\cap B_i\subseteq B_i$. By our choice, $i_b$ is the minimum index of $B_i$, we have  $i_{b}\leq j_{b}$. Then we conclude $i_{b}=j_{b}$. Hence,
 \begin{eqnarray*}
A_i \cap B_j  = \left \{ x_l \: | \: x_l \in A_i \:, \:   l\geq j_{b} \right \}=\left \{ x_l \: | \: x_l \in A_i\:, \: l\geq i_{b}   \right \} = A_i\cap B_i.
\end{eqnarray*}
Again by conditions $(a)$ and $(b)$, $A_i \cap B_j=A_i\cap B_i$ has size $t$. 
\vspace{0.2cm}

\noindent If $i_{a} \geq j_{a}$, we get
\begin{eqnarray*}
A_j \cap B_i & = \left \{ x_l \: | \: x_l \in B_i \: , \:  l \leq j_a \} \right \}\subseteq \left \{ x_l \: | \: x_l \in B_i \: , \:   l\leq i_a  \right \} = A_i\cap B_i.
\end{eqnarray*}
If $i_{a} \leq j_{a}$, we get
\begin{eqnarray*}
A_i \cap B_j & = \left \{ x_l \: | \: x_l \in B_i \: , \:  l \leq i_a \} \right \}\subseteq \left \{ x_l \: | \: x_l \in B_i \: , \:   l\leq j_a  \right \} = A_j\cap B_j.
\end{eqnarray*} 

\noindent 	One of $i_{a} \geq j_{a}$ and $i_{a} \leq j_{a}$ must hold. Then  either $A_j\cap B_i = A_i\cap B_i$ or $A_i\cap B_j = A_j\cap B_j$ holds, where $ i \neq j$. In any case, we conclude that $A_i\cap B_j = A_j\cap B_j = A_i\cap B_i = A_j\cap B_j$ for $t \geq 1$. It is a contradiction to the condition $(c')$. 
\end{proof}

\noindent Pick a permutation $\sigma \in S_n$ uniformly and independently, and let $E_i$ be the event that $\sigma$ properly separates $(A_i , B_i)$. Then the events $E_1 , E_2 , ... , E_k$ are mutually disjoint, and $\Pr[E_i ] = \binom{|A_i \cup B_i |}{|A_i - B_i |}^{-1} \binom{|B_i|}{|A_i \cap B_i |}^{-1}$ for $1 \leq i \leq k$. Then we conclude that

\begin{eqnarray*}
\Pr\left [\bigcup_{i=1}^{k}{E_i} \right ] = \sum_{i=1}^{k}{\Pr[E_i]} = \sum_{i=1}^{k}{\binom{|A_i \cup B_i |}{|A_i - B_i |}^{-1} \binom{|B_i |}{|A_i \cap B_i |}^{-1}} \leq 1
\end{eqnarray*}
as desired.

\section{Acknowledgement}
The authors are indebted to the anonymous referee who gave us valuable comments and pointed out a mistake in the former proof of Theorem \ref{thm:improve}.


\begin{thebibliography}{99}
\bibitem{ALON} N. Alon, H. Aydinian and H. Huang, \textit{Maximizing the number of nonnegative subsets}, SIAM J. Discrete Math., to appear, 2014. 
\bibitem{BOL} B. Bollob\'{a}s, \textit{On generalized graphs}, Acta Math. Acad. Sci. Hungar. \textbf{16} (1965) 447-452.
\bibitem{BORG13} P. Borg, \textit{The maximum product of sizes of cross $t$-intersecting uniform families}, submitted 2013.

\bibitem{EKR} P. Erd\H os, C. Ko and R. Rado, \textit{Intersection theorems for systems of finite sets}. Quart. J. Math. Oxford Ser., \textbf{12} (1961), 313-318.

\bibitem{FR76} P. Frankl, \textit{On Sperner families satisfying an additional condition}, J. Combin. Theory Ser. A
\textbf{20} 1-11 (1976).

\bibitem{FR78} P. Frankl, \textit{The Erd\H os-Ko-Rado theorem is true for $n=ckt$}, in Combinatorial Proc. Fifth
Hungarian Coll. Combinatorics, Keszthely, 1976, North-Holland, Amsterdam. 1978, 365-375.
\bibitem{FR82} P. Frankl, \textit{An extremal problem for two families of sets}, European J. Combin. \textbf{3} (1982) 125-127.
\bibitem{FR11} P. Frankl and N. Tokushige, \textit{On $r$-Cross intersecting Families of Sets}, Combinatorics, Probability and Computing, \textbf{20} 749-752.
\bibitem{FUREDI} Z. F\"{u}redi, \textit{Geometrical Solution of an Intersection Problem for Two Hypergraphs}, Europ. J. Combinatorics (1984) \textbf{5}, 133-136.

\bibitem{KNPV} Z. Kir\'{a}ly, Z. L. Nagy, D. P\'{a}lv\H{o}lgyi, M. Visontai, \textit{Fundamenta Informaticae} (1--4) (2012), 189--198.


\bibitem{LOVA} L. Lov\'{a}sz, Combinatorial Problems and Exercises, North-Holland, Amsterdam, New York, Oxford, 1979.
\bibitem{TAL} J. Talbot,  \textit{ A new Bollobas-type inequality and applications to $t$-intersecting families of sets}
, Discrete Math. \textbf{285} (2004), 349-353.
\bibitem{WIL}
R. Wilson,
\textit{The exact bound in the Erd\H os-Ko-Rado theorem}, Combinatorica \textbf{4}
(1984)  247--257.


\end{thebibliography}
\end{document}